\documentclass[11pt,leqno,oneside]{article}

\usepackage[utf8]{inputenc} %
\usepackage[T1]{fontenc}    %
\usepackage{microtype} 
\usepackage[english]{babel} %
\usepackage[lcgreekalpha]{stix2}          %

\usepackage[a4paper,margin=1in,marginpar=1in]{geometry} %
\usepackage{csquotes}       %

\usepackage[noblocks]{authblk}  %

\usepackage[shortlabels]{enumitem}

\usepackage{amsmath}   %
\usepackage{amsthm}    %
\usepackage{mathtools} %
\def\mapsto{\DOTSB\mathchar"39AD }

\usepackage[
backend=biber,
style=alphabetic,
giveninits=true,
maxalphanames=5,
minalphanames=5,
maxbibnames=99
]{biblatex}

\AtEveryBibitem{%
  \clearfield{note}%
  \clearfield{language}%
}

\usepackage{xcolor}                   %
\definecolor{green}{HTML}{2ECC71}
\definecolor{blue}{HTML}{3498DB}
\definecolor{red}{HTML}{E74C3C}
\definecolor{orange}{HTML}{FD6A02}
\usepackage[pdfpagelabels]{hyperref}  %
\hypersetup{%
  colorlinks,
  linktocpage,
  urlcolor  = blue,
  citecolor = green,
linkcolor = red}

\usepackage[capitalise,sort]{cleveref}

\AfterEndEnvironment{proof}{\noindent\ignorespaces} %
\makeatletter
\def\@endtheorem{\endtrivlist}
\makeatother
\Crefname{equation}{}{}
\Crefname{enumi}{}{}
\Crefname{conditioni}{Condition}{Conditions}
\Crefname{conditionalti}{Condition}{Conditions}
\newtheorem{theorem}{Theorem}[section]
\newtheorem*{theorem*}{Theorem}
\Crefname{theorem}{Theorem}{Theorems}
\newtheorem{lemma}[theorem]{Lemma}
\Crefname{lemma}{Lemma}{Lemmas}
\newtheorem{proposition}[theorem]{Proposition}
\Crefname{proposition}{Proposition}{Propositions}

\Crefname{corollary}{Corollary}{Corollaries}
\newtheorem{conjecture}[theorem]{Conjecture}
\Crefname{conjecture}{Conjecture}{Conjectures}

\Crefname{assumption}{Assumption}{Assumptions}
\theoremstyle{definition}
\newtheorem{definition}[theorem]{Definition}
\Crefname{definition}{Definition}{Definitions}

\Crefname{question}{Question}{Questions}
\theoremstyle{remark}
\newtheorem{remark}[theorem]{Remark}
\Crefname{remark}{Remark}{Remarks}

\Crefname{example}{Example}{Examples}
\newtheorem*{example*}{Example}

\numberwithin{equation}{section}

\usepackage{booktabs}

\DeclarePairedDelimiter{\paren}{\lparen}{\rparen}

\DeclarePairedDelimiter{\set}{\lbrace}{\rbrace}
\DeclarePairedDelimiter{\abs}{\lvert}{\rvert}

\DeclarePairedDelimiterXPP{\Exp}[1]{\exp}{\lparen}{\rparen}{}{#1}
\DeclarePairedDelimiterXPP{\Log}[1]{\log}{\lparen}{\rparen}{}{#1}
\DeclarePairedDelimiterXPP{\Inf}[1]{\inf}{\lbrace}{\rbrace}{}{#1}
\DeclarePairedDelimiterXPP{\Sup}[1]{\sup}{\lbrace}{\rbrace}{}{#1}

\DeclareMathOperator{\esp}{\mathbb{E}}

\DeclarePairedDelimiterXPP{\Esp}[1]{\esp}[]{}{#1}

\author[a]{Ronan HERRY\thanks{R.H gratefully acknowledges funding from the Centre Henri Lebesgue.}}
\author[b]{Dominique MALICET}
\author[a]{Guillaume POLY}
\affil[a]{IRMAR, Université de Rennes 1}
\affil[b]{LAMA, Université Gustave Eiffel}

\DeclareSourcemap{
  \maps[datatype=bibtex]{
    \map{
      \step[fieldsource=doi,final]
      \step[fieldset=url,null]
    }  
  }
}

\begin{document}
\title{A short proof of the strong three dimensional Gaussian product inequality}
 \maketitle
 \vspace{-2em}
\begin{abstract}
  We prove the strong form of the Gaussian product conjecture in dimension three.
  Our purely analytical proof simplifies previously known proofs based on combinatorial methods or computer-assisted methods, and allows us to solve the case of any triple of even positive integers which remained open so far.
\end{abstract}

\section{Introduction and main result}

In this note, we prove the following theorem.
\begin{theorem}\label{t:dim3}
  Let $(X_{1}, X_{2}, X_{3})$ be centered real Gaussian vector, and $p_{1}, p_{2}, p_{3} \in 2 \mathbb{N}$.
  Then,
  \begin{equation}\label{e:dim3}
    \Esp*{ X_{1}^{p_{1}} X_{2}^{p_{2}} X_{3}^{p_{3}}} \geq \Esp*{ X_{1}^{p_{1}}} \Esp*{ X_{2}^{p_{2}}} \Esp*{ X_{3}^{p_{3}}},
  \end{equation}
  with equality if and only if $X_{1}, X_{2}, X_{3}$ are independent.
\end{theorem}
Hence, our result completely solves the case $n =3$ of the strong form of the celebrated \emph{Gaussian product conjecture}.
For short, let us introduce the following notation.
\begin{definition}
  We say that $n \in \mathbb{N}$, and $p_{1}, \dots, p_{n} \in (0,\infty)$ satisfy the \emph{Gaussian product inequality}, and we write $\mathbf{GPI}_{n}(p_{1}, \dots, p_{n})$ provided for all real centered Gaussian vector $(X_{1}, \dots, X_{n})$:
  \begin{equation*}
    \Esp*{ \prod_{i=1}^{n} \abs{X_{i}}^{p_{i}} } \geq \prod_{i=1}^{n} \Esp*{ \abs{X_{i}}^{p_{i}}},
  \end{equation*}
  with equality if and only if $X_{1}, \dots, X_{n}$ are independent.
\end{definition}

\begin{conjecture}\label{gpc}
  For all $n \in \mathbb{N}$, and all $p_{1}, \dots, p_{n} \in 2 \mathbb{N}$, $\mathbf{GPI}_{n}(p_{1}, \dots, p_{n})$ holds.
\end{conjecture}

Despite having received considerable attention, the general case of the conjecture had, until now, remained wide open.
The previous state of the art regarding the Gaussian product conjecture was the following. 
\begin{theorem}\label{t:known}
  \begin{enumerate}[(a),wide,nosep]
    The following cases of \cref{gpc} are known.
    \item\label{t:known:dim2} For all $p_{1}, p_{2} \in 2\mathbb{N}$, $\mathbf{GPI}_{2}(p_{1}, p_{2})$.
    \item (\cite{Frenkel}) For all $n \in \mathbb{N}$, $\mathbf{GPI}_{n}(2, 2, \dots, 2)$.
      \item\label{t:known:dim3} (\cite{LCW}) For all $p \in 2 \mathbb{N}$, $\mathbf{GPI}_{3}(p, p, p)$.
      \item\label{t:known:sos} (\cite{RS}) For all $p \in 2 \mathbb{N}$, $\mathbf{GPI}_{3}(p,6,4)$ and $\mathbf{GPI}_{2}(p,2,2,2)$.
    \item (\cite{RS2}) For all $p$ and $q \in 2 \mathbb{N}$, $\mathbf{GPI}_{3}(2,p,q)$.
  \end{enumerate}
\end{theorem}

The above results are obtained through sophisticated methods.
In particular, \cite{LCW} relies on a heavily combinatorial approach in connection with the theory of Gaussian hypergeometric functions; while \cite{RS,RS2} is a computer-assisted method based on the SOS algorithm which provides an explicit expansion of a positive multivariate polynomial into a sum of squared quantities.
On the contrary, our approach is purely analytical and combines an optimization procedure through the use of Lagrange multipliers with Gaussian analysis.
Our contribution not only drastically simplifies the proof of the known cases in dimension three (\cref{t:known} \cref{t:known:dim3,t:known:sos}), but it also enables us to fully resolve the three dimensional case, that is to say for every choice of even integer exponents.

\begin{remark}
Let us also mention that, for all $n \in \mathbb{N}$ and $p_{1}, \dots, p_{n} \in 2 \mathbb{N}$, \cite{Arias} establishes the \emph{complex} counterpart of the conjecture; while \cite{MNPP} derives a variant of the inequality involving \emph{Hermite polynomials}.
\cite{Wei} proves $\mathbf{GPI}_{n}(p_{1}, \dots, p_{n})$ for all $n \in \mathbb{N}$ and $p_{1}, \dots, p_{n} \in (-1,0)$.
Several authors also derive weaker form of \cref{gpc} by considering only Gaussian vectors with additional assumptions on the covariance matrix.
Among others, let us quote the two recent contributions: \cite{Royen} for the case of positive correlations, and \cite{Ouimet} for multinomial covariances.
\end{remark}

\section{Proof of the main result}
In the rest of the paper, we fix a probability space $(\Omega,\mathcal{F},\mathbb{P})$ supporting an independent sequence $(G_k)_{k\in\mathbb{N}}$ of centered normalized Gaussian variables on $\Omega$.
In the following, $\Sigma = (\sigma_{i,j})$ is a real symmetric non-negative matrix of size $n$.
To $\Sigma$, we associate a centered Gaussian vector $\vec{X}=(X_1,\ldots,X_n)$ with covariance matrix $\Sigma$ by setting $\vec{X}= \Sigma^{1/2} \vec{G}$ where $\vec{G}=(G_1,\ldots,G_n)$.
Let $p_{1}, \dots, p_{n} \in 2 \mathbb{N}^{*}$, and $h(x_1,\dots,x_n)=x_1^{p_1} \dots x_{n}^{p_{n}}$.
  Our strategy consists in studying the points where the map $\Phi \colon \Sigma \mapsto \Esp{h(X_1,\dots,X_n)}$ reaches its minimum.
  Our argument allows us to characterize those minimal points for $n= 2$ or $3$.
  Using  Wick formula \cite[Thm.\ 1.28]{Janson}, it is readily checked that $\Phi$ is polynomial in the entries of $\Sigma$.
  We shall need the following standard lemma.
  We recall a proof for the sake of self-containedness.
\begin{lemma}\label{lemme-technique}
Let $(X_1,\cdots,X_n)$ be a Gaussian vector that is centred with covariance matrix $\Sigma$ non necessarily invertible.
Then it holds:
\begin{align}
    & \Esp*{X_ih(X_1,\ldots, X_n)}=\sum_{j=1}^n \sigma_{i,j}\Esp*{\partial_jh(X_1,\ldots,X_n)}, \qquad i \in \set{1,\dots, n}\label{ipp} ; \\
    & \frac{\partial}{\partial\sigma_{i,j}}\Esp*{h(X_1,\ldots,X_n)}=\Esp*{\partial_{x_i} \partial_{x_j}h(X_1,\ldots,X_n)}, \qquad i \ne j \in \set{1,\dots, n} \label{derivative}.
  \end{align}
\end{lemma}
\begin{proof}[{Proof of \cref{lemme-technique}}]
  In view of Wick formula, \cref{ipp,derivative} are equalities involving polynomials in the entries of $\Sigma$.
  It is thus sufficient to establish them for an invertible $\Sigma$.
  In this case, let us write $f_{\Sigma}$ for the density distribution associated with $\vec{X}$.
  A direct computation yields
  \begin{equation*}
    x_{i} f_{\Sigma} + \sum_{j=1}^{n} \sigma_{i,j} \partial_{x_{j}} f_{\Sigma} = 0, \qquad i \in \set{1,\dots,n}.
  \end{equation*}
\cref{ipp} readily follows.
In order to prove \cref{derivative}, consider the Fourier transform of $f_{\Sigma}$:
\begin{equation*}
  \widehat{f_\Sigma}(x_1,\cdots,x_n)=\Exp*{-\frac{1}{2} \sum_{i,j=1}^n x_i x_j \sigma_{i,j}}, \qquad (x_{1}, \dots, x_{n}) \in \mathbb{R}^{n}.
\end{equation*}
Differentiating this formula, we obtain for $i\neq j$:
\begin{equation*}
  \widehat{\partial_{\sigma_{i,j}} f_\Sigma} = \partial_{\sigma_{i,j}} \widehat{f_\Sigma}=-x_i x_j \widehat{f_\Sigma}=\widehat{\partial_{x_i}\partial_{x_j}f_{\Sigma}}.
\end{equation*}
Since the Fourier transform is into, we get that $\partial_{\sigma_{i,j}} f_{\Sigma} = \partial_{x_{i}} \partial_{x_{j}} f_{\Sigma}$, from which \cref{ipp} readily follows.
\end{proof}

In order to highlight the line of reasoning we use in the proof of \cref{t:dim3}, and for the sake of completeness, let us first give a short proof of \cref{t:known} \cref{t:known:dim2}.
\begin{proof}[Proof of {\cref{t:known} \cref{t:known:dim2}}]
  Fix $p_1$ and $p_2 \in 2 \mathbb{N}^{*}$.
  We want to prove that if $(X_1,X_2)$ is a centered real Gaussian vector then $\Esp{X_1^{p_1} X_2^{p_2}}\geq \Esp{X_1^{p_1}} \Esp{X_2^{p_2}}$, with equality if and only if  $X_1$ and $X_2$ are independent.
  By homogeneity it is enough to prove the statement when $X_1$ and $X_2$ are normalized, in which case the right term of the inequality depends only on $p_1$ and $p_2$.
  Setting, for $t$ in $[-1,1]$, $\Phi(t)=\Esp{X_1^{p_1}X_2^{p_2}}$ where $(X_1,X_2)$ is the Gaussian vector associated to
\begin{equation*}
  \Sigma=
  \begin{pmatrix}
    1 & t\\ t & 1
  \end{pmatrix},
\end{equation*}
the claim is equivalent to the fact that $\Phi$ reaches its unique minimum at $t=0$.
From \cref{derivative}, $\Phi'(t)=p_1 p_2 \Esp{X_1^{p_1-1}X_2^{p_2-1}}$ and $\Phi''(t)=p_1(p_1-1) p_2(p_2-1) \Esp{X_1^{p_1-2}X_2^{p_2-2}}$.
In particular, $\Phi''>0$ and $\Phi'(0)=0$.
Consequently, $0$ is a critical point of a strictly convex function, and thus it is the unique global minimizer of $\Phi$, from which the result follows.
\end{proof}

\cref{t:dim3} follows from the recursive argument below; the corresponding initialization is given by \cref{t:known} \cref{t:known:dim2}.
\begin{proposition}
  Let $p_{1}, p_{2}, p_{3} \in 2 \mathbb{N}^{*}$.
  If $\mathbf{GPI}_{3}(p_{1} - 2, p_{2}, p_{3})$, then $\mathbf{GPI}_{3}(p_{1}, p_{2}, p_{3})$.
\end{proposition}

\begin{proof}

Let $\mathcal{C}$ be the set of real symmetric positive matrices of size $3$ with $1$ on the diagonal, namely
\begin{equation*}
  \mathcal{C}= \set*{ \Sigma=\begin{pmatrix} 1 & a & b \\ a & 1 & c \\ c & b & 1 \end{pmatrix} : |a|,|b|,|c|\leq 1,\, \det(\Sigma)\geq 0   }.
\end{equation*}
We identify $\mathcal{C}$ with a compact subset of $\mathbb{R}^{3}$. With this notation, $\mathbf{GPI}_{3}(p_{1}, p_{2}, p_{3})$ turns out to be equivalent to the fact that $\Phi$ attains its unique minimum on $\mathcal{C}$ at $I_{3}$.
Since $\mathcal{C}$ is compact and $\Phi$ continuous, $\Phi$ has a global minimum at some possibly non-unique 
\begin{equation*}
  \Sigma_0 = \begin{pmatrix} 1 & a & b \\ a & 1 & c \\ c & b & 1 \end{pmatrix} \in \mathcal{C}.
\end{equation*}
We prove that $\Sigma_0=I_3$.
We split the argument in three cases, depending on the location of $\Sigma_{0}$ in $\mathcal{C}$.

\paragraph{Case 1.}
We assume that $\Sigma_{0}$ is in the interior on $\mathcal{C}$.
 This means that $\det(\Sigma_{0}) > 0$ and $\abs{a}, \abs{b}, \abs{c} < 1$.
In this case, $\Sigma_0$ is a critical point of $\Phi$.
Write
\begin{equation*}
  U = X_{1}^{p_{1}-1}X_{2}^{p_{2}-1}X_{3}^{p_{3}-1}.
\end{equation*}
According to \cref{derivative}
\begin{equation}\label{Phiderivative}
  \begin{cases}
    \partial_{a} \Phi(\Sigma_0)=p_1 p_2 \Esp{X_{3} U},\\
    \partial_{b} \Phi(\Sigma_0)=p_1 p_3 \Esp{X_{2} U},\\
    \partial_{c} \Phi(\Sigma_0)=p_2 p_3 \Esp{X_{1} U}.
\end{cases}
\end{equation}
Thus, $\Esp{X_{1} U}=\Esp{X_{2} U}=\Esp{X_{3} U}=0$.
On the other hand, let
\begin{equation*}
  V = X_{1}^{p_{1} -1} X_{2}^{p_{2}} X_{3}^{p_{3}}.
\end{equation*}
Thus, by \cref{ipp} and the fact that the derivatives vanish,
\begin{equation*}
  \begin{split}
    \Phi(\Sigma_0)&=\Esp{X_{1} V}
                \\&=(p_1-1)\Esp{X_1^{p_1-2}X_2^{p_2}X_3^{p_3}}+p_2a \Esp{X_{3} U}+ p_3 b \Esp{X_{2} U}
             \\&=(p_1-1)\Esp{X_1^{p_1-2}X_2^{p_2}X_3^{p_3}}
  \end{split}
\end{equation*}
In view of $\mathbf{GPI}_{3}(p_{1}-2, p_{2}, p_{3})$, we thus get
\begin{equation*}
  \Phi(\Sigma_{0}) \geq (p_1-1)\Esp{X_1^{p_1-2}}\Esp{X_2^{p_2}}\Esp{X_3^{p_3}}=\Esp{X_1^{p_1}}\Esp{X_2^{p_2}}\Esp{X_3^{p_3}}=\Phi(I_3).
\end{equation*}
Since $\Sigma_0$ is a minimizer, we actually have that $\Phi(\Sigma_{0}) = \Phi(I_{3})$.
In particular, this means that we are in the equality case of $\mathbf{GPI}_{3}(p_{1}-2, p_{2}, p_{3})$.
If $p_{1}  > 2$, this shows mediately that $\Sigma_{0} = I_{3}$.
Similarly, if $p_{2} > 2$ or $p_{3} > 2$, we conclude in the same way.
If $p_{1} = p_{2} = p_{3} = 2$, using \cref{t:known} \cref{t:known:dim2}, we deduce that the components of $(X_{1}, X_{2}, X_{3})$ are pairwise independent.
Since the vector is Gaussian, the conclusion follows.

\paragraph{Case 2.}
We assume that $|a|, |b|, |c|<1$ and $\det(\Sigma_0)=0$.

$\Sigma_0$ is a priori not a critical point of $\Phi$.
Since $\Sigma_0$ is a global minimizer on $\mathcal{C}$ it is also a minimizer of $\Phi$ on the surface
\begin{equation*}
  S= \set*{\Sigma=\begin{pmatrix} 1 & a & b \\ a & 1 & c \\ c & b & 1 \end{pmatrix} : \det(\Sigma)=0 }.
\end{equation*}
By the Lagrange multiplier theorem, we conclude that $\vec{\nabla}\Phi(\Sigma_0)$ and $\vec{\nabla}(\det)(\Sigma_0)$ are colinear (where $\vec{\nabla}=(\frac{\partial}{\partial a},\frac{\partial}{\partial b},\frac{\partial}{\partial c})$).
We have already computed $\vec{\nabla}\Phi(\Sigma_0)$ in \cref{Phiderivative}, and we have:

\begin{lemma} 
	$\vec{\nabla}(\det)(\Sigma_0)=(\alpha_1\alpha_2,\alpha_1\alpha_3,\alpha_2\alpha_3)$, where $(\alpha_1, \alpha_2, \alpha_3)$ is some non zero vector of $\ker(\Sigma_0)$.
\end{lemma}

\begin{proof}
  Write $A = 2 \mathrm{adj}(\Sigma_{0})$ where $\mathrm{adj}$ stands for the adjugate matrix
  Since $\det(\Sigma_0)=0$, $\mathrm{rank}(\Sigma_0)\leq 2$, and since $|a|, |b|, |c|<1$, two columns of $\Sigma_0$ cannot be proportional so $\mathrm{rank}(\Sigma_0)=2$.
  This implies that $A$ has rank $1$, thus $A = \alpha^{T} \alpha$ where $\alpha=(\alpha_1,\alpha_2,\alpha_3) \in \ker(\Sigma_0) \setminus \{0\}$.
  By Jacobi's formula, we have that $\vec{\nabla}(\det)(\Sigma_0)=(A_{1,2}, A_{1,3}, A_{2,3})$.
\end{proof}

 We deduce that there exists a real number $k$ such that
\begin{equation}\label{multipliers}
  \begin{cases}
    \partial_{a} \Phi(\Sigma_0)=k\alpha_1\alpha_2,\\
    \partial_{b} \Phi(\Sigma_0)=k\alpha_1\alpha_3,\\
    \partial_{c} \Phi(\Sigma_0)=k\alpha_2\alpha_3.
\end{cases} .\end{equation}

Since $(\alpha_1,\alpha_2,\alpha_3)$ belongs to $\ker(\Sigma_0)$, $\alpha_1X_1+\alpha_2 X_2+\alpha_3 X_3=0$ almost surely.
From \cref{Phiderivative},
\begin{equation*}
  p_1\alpha_1\frac{\partial \Phi}{\partial c}(\Sigma_0)+p_2\alpha_2\frac{\partial \Phi}{\partial b}(\Sigma_0)+p_3\alpha_3\frac{\partial \Phi}{\partial a}(\Sigma_0)=p_1p_2p_3\Esp*{U\paren*{\alpha_1X_1+\alpha_2 X_2+\alpha_3 X_3}}=0.
\end{equation*}
Thus by reporting in \cref{multipliers},
\begin{equation*}
  0=(p_1+p_2+p_3)k\alpha_1\alpha_2\alpha_3.
\end{equation*}
If $k=0$, then \cref{multipliers} gives that $\Sigma_0$ is a critical point of $\Phi$, and as in Case 1 we obtain that $\Sigma_0=I_3$, which contradicts $\det(\Sigma_0)=0$.
If one of the $\alpha_i$ is zero, say $\alpha_1$, then  $\alpha_2X_2+\alpha_3X_3=0$, so $X_3$ and $X_2$ are proportional, and since they are normalized, $X_3=\pm X_2$, which contradicts $|c|<1$.
Hence, Case 2 cannot happen.

\paragraph{Case 3.}
We assume that $\{|a|,|b|,|c|\}\cap\{1\}\neq \emptyset$. Say for example that $|c|=1$.
That implies that $X_3=\pm X_2$ and so by the two dimensional case \cref{t:known} \cref{t:known:dim2},
\begin{equation*}
  \Phi(\Sigma_0)=\Esp{{X_1}^{p_1}{X_2}^{p_2+p_3}}\geq \Esp{{X_1}^{p_1}}\Esp{{X_2}^{p_2+p_3}} > \Esp{{X_1}^{p_1}}\Esp{{X_2}^{p_2}}\Esp{{X_2}^{p_3}}=\Phi(I_3).
\end{equation*}
In particular $\Sigma_0$ is not a minimizer.
It is a contradiction, and Case 3 cannot either happen.

\paragraph{Conclusion.}
We obtain that the only minimizer of $\Phi$ is $I_3$ which concludes the proof.
\end{proof}

\printbibliography%

\end{document}